\documentclass[11pt]{amsart}
\usepackage{amsfonts}
\usepackage{amsmath,amssymb,amsthm,latexsym,graphicx,graphics}
\usepackage{multicol}
\usepackage{color}
\usepackage{url,hyperref}
\usepackage{euscript}
\usepackage{wrapfig}
\usepackage{caption}
\def \R{\mathbb{R}}

\newtheorem{theorem}{Theorem}[section]
\newtheorem{cor}[theorem]{Corollary}

\newtheorem{lema}[theorem]{Lemma}
\newtheorem{rem}[theorem]{Remark}
\newtheorem{remark}[theorem]{Remark}

\newtheorem{prop}[theorem]{Proposition}
\newtheorem{thm}{Theorem}

\usepackage{url,hyperref}
\hypersetup{
    bookmarks=true,         
    unicode=false,          
    pdftoolbar=true,        
    pdfmenubar=true,        
    pdffitwindow=false,     
    pdfstartview={FitH},    
    pdftitle={My title},    
    pdfauthor={Author},     
    pdfsubject={Subject},   
    pdfcreator={Creator},   
    pdfproducer={Producer}, 
    pdfkeywords={keyword1} {key2} {key3}, 
    pdfnewwindow=true,      
    colorlinks=true,       
    linkcolor=black,          
    citecolor=green,        
    filecolor=magenta,      
    urlcolor=blue         
}

\usepackage{euscript,color}

\definecolor{red}{rgb}{1,0,0}

\hypersetup{
    bookmarks=false,         
    unicode=false,          
    pdftoolbar=true,        
    pdfmenubar=true,        
    pdffitwindow=false,     
    pdfstartview={FitH},    
    pdftitle={My title},    
    pdfauthor={Author},     
    pdfsubject={Subject},   
    pdfcreator={Creator},   
    pdfproducer={Producer}, 
    pdfkeywords={keyword1} {key2} {key3}, 
    pdfnewwindow=true,      
    colorlinks=true,       
    linkcolor=black,          
    citecolor=green,        
    filecolor=magenta,      
    urlcolor=red         
}

\def\R{\mathbb{R}}

\def\Hol{\mathrm{Hol}}
\def\meti{\left\langle}
\def\metd{\right\rangle}

\def\crn{\mathcal{R}^{\bot}}

\begin{document}

\title[The normal holonomy of $CR$-submanifolds]
{The normal holonomy of $CR$-submanifolds}

\author[A. J. Di Scala]{Antonio J. Di Scala}
\author[F. Vittone]{Francisco Vittone}


\date{\today}
\begin{abstract}
We study the normal holonomy group, i.e. the holonomy group of the normal connection, of a $CR$-submanifold of a complex space form.
We show that the normal holonomy group of a coisotropic submanifold acts as the holonomy representation of a Riemannian symmetric space.
In case of a totally real submanifold we give two results about reduction of codimension. We describe explicitly the action of the normal holonomy in the case in which the totally real submanifold is contained in a totally real totally geodesic submanifold. In such a case we prove the compactness of the normal holonomy group.
\end{abstract}

\maketitle

\section{Introduction}
The objective of this paper is to study the normal holonomy group of $CR$-submanifolds of complex space forms.

For submanifolds of $\R^n$ or more generally of real space forms,
a fundamental result is the \textit{Normal Holonomy Theorem}\cite{Ol90}. It asserts roughly that the non-trivial component of the action of the restricted normal holonomy group acts on any normal space as the isotropy representation of a Riemannian symmetric space (called $s$-representation for short). The Normal Holonomy Theorem is a very important tool for the study of submanifold geometry, especially in the context of  submanifolds with ``simple extrinsic geometric invariants'', like isoparametric and homogeneous submanifolds (see \cite{BCO} for an introduction to this subject).
Moreover, the Normal Holonomy Theorem has reveled to have important consecuences in the study of intrinsic riemannian geometry,  as it can be seen, for example, in the very important role it plays in the geometric proof of Berger's Theorem \cite{O5}. So a natural question is whether it can be generalized to submanifolds of other ambient spaces, and in particular to submanifolds of complex space forms.

The main tool in the proof of Olmos's theorem is the simplicity of the Ricci equation in real space forms.
Therefore $CR$-submanifolds constitute the natural family of submanifolds to explore the validity of the Normal Holonomy Theorem, and they include some important types of submanifolds such as complex, isotropic (also called totally real or anti-invariant), coisotropic (also called CR-generic) and Lagrangian. (see section \ref{hopf}).

$CR$-submanifolds have been widely studied, see for example \cite{Be1, B2, ChenCR, DO, YK}. Isotropic or totally real submanifolds are those on which the complex structure $J$ maps the tangent space into the normal space at each point. In contraposition, coisotropic submanifolds are those on which $J$ maps the normal space into the tangent space. Lagrangian submanifolds are those which are at the same time isotropic and coisotropic.

\vspace{0.5cm}
%
%
%

Our main result shows that the Normal Holonomy Theorem holds for coisotropic submanfolds of complex space forms.

\begin{thm}\label{teoremagen}
Let $M$ be a coisotropic submanifold of a complex space form $\mathbb{S}_{c}$. Then the restricted normal holonomy group of $M$ acts on the normal space as the holonomy representation of a Riemannian symmetric space i.e. a flat factor plus a s-representation.
\end{thm}

For Lagrangian submanifolds the complex structure $J$ of the ambient space form induces a natural isomorphism between the normal and the Riemannian holonomy groups. Therefore we obtain the following important consequence.

\begin{cor}\label{cor:RicciFlat} A Ricci flat Lagrangian submanifold of a complex space form $\mathbb{S}_{c}$ has non-exceptional Levi-Civita holonomy, i.e., it is either flat or the restricted holonomy group of its Levi-Civita connection is $SO(TM)$.
\end{cor}

Non-full totally real submanifolds and Lagrangian submanifolds play also an important role on the study of extrinsically symmetric submanifolds. In \cite{Na}, it was shown that extrinsically symmetric submanifolds of complex space forms are complex submanifolds, totally real submanifolds contained in a totally real totally geodesic submanifold, or lagrangian submanifolds of a totally geodesic complex submanifold.

In Section \ref{Section:TotallyReal} we explore the normal holonomy group of totally real submanifolds of complex space forms. We start with an example showing that the strategy we used for coisotropic submanifolds can not be adapted to this case. In particular, we characterize the so called holomorphic circles \cite[page 8, Definition]{AMU} (also called K\"ahler-Frenet curves \cite[Introduction]{MT}) as those curves of the complex projective space whose pull-back to the sphere via the Hopf fibration has flat normal bundle.
We give two results about reduction of codimension. For totally real submanifolds of totally geodesic totally real submanifolds of a complex space form we give an explicit description of the action of its normal holonomy group.
It turns out that the normal holonomy group is compact but it does not act, in general, as in Olmos' holonomy theorem.\\

We end the paper with an observation missed in \cite{AD} about the restricted normal holonomy of a complex submanifold.
\begin{thm}\label{teo:normalComplejo} Let $M$ be a full (non necessarily complete) complex submanifold of a complex space form and let $\Hol^0_p(M,\nabla^{\perp})$ be the restricted holonomy group of the normal connection. Then $\Hol^0_p(M,\nabla^{\perp})$ acts on the normal space $\nu_p(M)$ as the isotropy representation of a (non necessarily irreducible) Hermitian symmetric space without flat factor.
\end{thm}
In \cite{AD} the above result was proved under the additional hypothesis that either the normal holonomy group acts irreducibly or the second fundamental form has no nullity. Theorem \ref{teo:normalComplejo} plays an important role in the local classification of normal holonomies given in \cite{DV15} (see \cite{CDO} for a global result).


\section{Preliminaries and basic facts}
\subsection{Complex space forms and Hopf fibrations}
\label{hopf1}
 Let $\mathbb{S}_{c}$ be a complex space form of holomorphic sectional curvature $c$. For the sake of simplicity we shall assume that $\mathbb{S}_c$ is one of the standard models, that is, the complex euclidean space $\mathbb{C}^n$ if $c=0$, the complex projective space $\mathbb{C}P^{n}$ if $c=4$ or the complex hyperbolic space $\mathbb{C}H^{n}$ if $c=-4$. However all the results here are valid for arbitrary $c$.

  Denote by $J$ the complex structure, by $\meti\; , \;\metd$ the standard metric and by $\overline{\nabla}$ the Levi-Civita connection on $\mathbb{S}_c$.

We will now introduce the Hopf fibrations for the complex hyperbolic and projective spaces.

If $z=(z_0,z_1,\cdots,z_n),\ w=(w_0,w_1,\cdots,w_n)\in\mathbb{C}^{n+1}$, define $$\left\langle z,w \right\rangle=Re\left(\sum_{i=0}^{n}z_i\overline{w}_i\right); \ \ \ \left\langle z,w \right\rangle_{1}=Re\left(-z_0\overline{w}_0+\sum_{i=1}^{n}z_i\overline{w}_i\right)$$
Then $\left\langle \;,\; \right\rangle$ is the standard inner product on $\mathbb{C}^{n+1}$, which can be identified with the Euclidean space $\mathbb{R}^{2n+2}$.

 On the other hand, $\left\langle \;,\; \right\rangle_{1}$ is a scalar product of signature $2$ on $\mathbb{C}^{n+1}$. We will denote by $\mathbb{C}^{n+1}_1$ the complex vector space $\mathbb{C}^{n+1}$ with this scalar product. Then $\mathbb{C}^{n+1}_{1}$ can be identified with the standard semi-Euclidean space $\mathbb{R}^{2n+2}_{2}$.

For $c=4$, denote by $\overline{N}_c$ the $(2n+1)$-dimensional sphere in $\mathbb{C}^{n+1}$, that is $$\overline{N}_{4}=S^{2n+1}=\{z\in\mathbb{C}^{n+1}\,:\,\meti z,z\metd=1\}$$ and for $c=-4$ denote by $\overline{N}_c$  the Lorentzian pseudo-hyperbolic space (or anti-De Sitter space) $H^{n+1}_1$ in $\mathbb{C}^{n+1}_1$, that is,
$$\overline{N}_{-4}=H^{n+1}_1=\{z\in \mathbb{C}^{n+1}\, :\,\meti z,z\metd_{1}=-1\}.$$

Recall that $H^{n+1}_1$ is a Lorentzian real space form of constant sectional curvature $K=-1$ (see \cite[Prop. 29, page 113]{ON}).

The one-parameter group $U(1)=\{z=e^{i\theta}:\theta\in\mathbb{R}\}$ acts by multiplication on $\overline{N}_c$ and $\mathbb{S}_c=\overline{N}_c/U(1)$. Moreover, the standard projection $$\pi_c:\overline{N}_c\to \mathbb{S}_c$$ is a principal fiber bundle, called \textsl{Hopf fibration}.

Let $\eta_{p}:=p$ be the position vector field on $\overline{N}_c$ and let $V_p$ and $H_p$ be the vertical and horizontal subspaces associated to $\pi_c$ at $p$ respectively. That is, $V_p=T_{p}(\pi_c^{-1}(\pi_c(p))$ and $H_p=(V_p)^{\bot}\subset T_p \overline{N}_c$. Then $$V_p=\mathrm{span}_{\mathbb{R}}\{J\eta_p\};\ \ H_p\equiv T_{\pi_c(p)}\mathbb{S}_c.$$
$J\eta$ is called the \textsl{Hopf vector field} of $\overline{N}_c$.
Observe that $H_p$ is a $J$-invariant subspace and $d\pi_c$ identifies $J_{|H_p}$ with the complex structure $J$ of $\mathbb{S}_c$. Moreover, $\pi_4$ is a Riemannian submersion and $\pi_{(-4)}$ is a pseudo Riemannian one, and in both cases $H$ defines a Riemannian subbundle of  $T\overline{N}_c$.

Denote by $\nabla'$ the Levi-Civita connection of $\overline{N}_c$ and by $g_c$ the metric on $\overline{N}_c$ induced from the corresponding inner product on the ambient complex space.

For a vector field $X$ in $\mathbb{S}_c$ we will always denote by $\hat{X}$ its \textbf{horizontal lift} to $\overline{N}_c$, i.e., $\hat{X}$ is the only horizontal vector field in $\overline{N}_c$, $\pi_c$-related to $X$.
Then from O'Neil formulas for a submersion one gets that for each $X,\ Y\in \mathfrak{X}(\mathbb{S}_c)$,
\begin{equation}\nabla'_{\hat{X}}\hat{Y}=\widehat{(\overline{\nabla}_X Y)} + g_c(X,JY)J\eta
\label{lift1}
\end{equation}
\begin{equation}
 \nabla'_{J\eta}\hat{X}=\nabla'_{\hat{X}}J\eta=J\hat{X}=\widehat{JX}.
 \label{lift2}
 \end{equation}
(cf. \cite{NT})

\subsection{$CR$-submanifolds}

\label{hopf}

A submanifold $M$ of $\mathbb{S}_c$ (or more generally, of a K\"ahlerian manifold) is called a \textsl{$CR$-submanifold} if there exists a differentiable distribution $\mathcal{D}$ on $M$ such that for each $x\in M$,  $D_x$ is a complex subspace of $T_x \mathbb{S}_c$, i.e., $J\mathcal{D}_x=\mathcal{D}_x$, and the orthogonal distribution $\mathcal{D}^{\bot}\subset TM$ is anti-invariant, i.e., $J\mathcal{D}^{\bot}_x$ is normal to $M$.

There are three  particular cases of $CR$-submanifolds we are interested in. If $\mathcal{D}_x=T_xM$, then $M$ is a \textsl{complex} submanifold of $\mathbb{S}_c$.

 If on the contrary $\mathcal{D}_x=\{0\}$, i.e. $JT_xM\subset \nu_xM$ for each $x$,  then $M$ is a \textsl{totally real} ( also called anti-invariant or isotropic) submanifold of $\mathbb{S}_c$.

Finally, if $\mathrm{dim}\; \mathcal{D}^{\bot}_x=\mathrm{dim}\; \nu_x M$, and consequently $J\nu_x M\subset T_xM$, $M$ is called a \textsl{coisotropic} ( also called generic $CR$-submanifold) of $\mathbb{S}_c$.

A submanifold which is both totally real and coisotropic, i.e., $JT_xM=\nu_xM$ is called a \textsl{Lagrangian submanifold} of $\mathbb{S}_c$.

For general facts about $CR$-submanifolds of K\"ahler manifolds see for example  \cite{Be1, B2, ChenCR, YK, DO}.
\vspace{0.5cm}

We will now introduce some preliminaries on the general theory of submanifolds of a complex space form and state how the geometry of a submanifold $M$ of the complex projective or hyperbolic space relates with that of its pull-back via the Hopf fibration.

Let $M$ be a Riemannian submanifold of $\mathbb{S}_c$. Denote by $\nabla$ the Levi-Civita connection of $M$ and by $\nabla^{\bot}$ the normal connection on the normal bundle $\nu M=(TM)^{\bot}$. Let $\alpha$ and $A$ be the second fundamental form and shape operator of $M$ respectively. They are defined, taking tangent and normal components with respect to the decomposition $T\mathbb{S}_{c|M}=TM\oplus\nu M$ by the Gauss and Codazzi formulas
\begin{equation}
\overline{\nabla}_X Y=\nabla_X Y+ \alpha(X,Y),\qquad \overline{\nabla}_X\xi=-A_{\xi}X+\nabla^{\bot}_X\xi
\label{gausscodazzi}
\end{equation}
and related by $\meti\alpha(X,Y),\xi\metd=\meti A_{\xi}X,Y\metd$, for any tangent vector fields $X$ and $Y$ to $M$ and any normal vector field $\xi$.

Denote by $\overline{R}^{c}$ the Riemannian curvature tensor of $\mathbb{S}_c$. Recall that if $X,Y\in \mathfrak{X}(\mathbb{S}_c)$ then
\begin{equation}
\overline{R}^{c}_{X,Y}=\frac{1}{4}c(X\wedge Y+JX\wedge JY-2\meti JX,Y\metd J)
\label{curvatura}
\end{equation}
where $X\wedge Y(Z)=\meti Y,Z\metd X-\meti X,Z\metd Y$.

Let $R$ and $R^{\bot}$ be the Riemannian and the normal curvature tensors of $M$ respectively. Then for $X,Y,Z$ tangent to $M$ and $\xi, \zeta$ normal to $M$, the well known equations of Gauss, Codazzi and Ricci hold:

\begin{equation}
\meti \overline{R}^c_{X,Y}Z,W\metd =\meti R_{X,Y}Z,W\metd + \meti \alpha(X,Z),\alpha(Y,W)\metd -\meti \alpha(X,W),\alpha(Y,Z)\metd
\label{gauss}
\end{equation}
\begin{equation}
(\overline{R}^{c}_{X,Y}Z)^{\bot}=(\nabla^{*}_X\alpha)(Y,Z)-(\nabla^{*}_{Y}\alpha)(X,Z)
\label{codazzi}
\end{equation}
\begin{equation}
\meti \overline{R}^{c}_{X,Y}\xi,\zeta\metd=\meti R^{\bot}_{X,Y}\xi,\zeta\metd - \meti [A_{\xi},A_{\zeta}]X,Y\metd.
\label{ricci}
\end{equation}
where $\nabla^{*}$ is the connection $\nabla\oplus \nabla^{\bot}$ on the vector bundle $T\mathbb{S}_{c|M}$.

\vspace{0.5cm}

Assume now that $M\subset \mathbb{S}_c$ with $c=4$ or $c=-4$.

Set $\hat{M}=\pi^{-1}(M)$ and $\hat{\pi}=\pi_{c|\hat{M}}$ where $\pi_c$ is the Hopf fibration introduced in the previous section. Then  $\hat{\pi}:\hat{M}\to M$ is a principal $U(1)$-bundle. If $c>0$, i.e. $M\subset \mathbb{C}P^n$, then  $\hat{M}$ is a Riemannian submanifold of the sphere $\overline{N}_4=S^{2n+1}$ and $\hat{\pi}$ is a Riemannian submersion. If $c<0$, i.e. $M\subset \mathbb{C}H^n$, then $\hat{M}$ is a Lorentzian submanifold of $\overline{N}_{(-4)}=H^{n+1}_1$ and $\hat{\pi}$ a pseudo-Riemannian submersion (observe that for $c=-4$ one has $g_{(-4)}(J\eta,J\eta)=-1$). Along this paper we will call $\hat{M}$ the {\bf pull-back} of $M$.

 The vertical subspace $\hat{V}_p$ of $\hat{\pi}$ coincides with $V_{p}=\mathbb{R}J\eta_p$ and the horizontal subspace $\hat{H}_p=H_p\cap T_p\hat{M}$ is isometric via $d \hat{\pi}_p$ with $T_{\hat{\pi}(p)}M$.

If $X,Y$ are tangent vector fields to $M$ and $\xi$ is a normal vector field to $M$, their horizontal lifts are respectively tangent and normal to $\hat{M}$.

Denote by $\hat{\nabla}$, $\hat{\nabla}^{\bot}$, $\hat{\alpha}$ and $\hat{A}$ the Levi-civita and normal connections respectively and the second fundamental form and shape operator of $\hat{M}$. Then from equations (\ref{lift1}) and (\ref{lift2}) it is not difficult to obtain the following equations (recall that the hat $\hat{\cdot}$ always indicates the horizontal lift of a vector):

\begin{equation}
\hat{\nabla}_{\hat X}\hat{Y}=\widehat{\nabla_X Y}+\meti X, JY\metd J\eta,\qquad \hat{\alpha}(\hat{X},\hat{Y})=\widehat{\alpha(X,Y)}.
\label{lift3}
\end{equation}

\begin{equation}
\hat{A}_{\hat{\xi}}\hat{X}=\widehat{A_{\xi}X}-\meti X,J\xi\metd J\eta,\qquad \hat{\nabla}^{\bot}_{\hat{X}}\hat{\xi}=\widehat{\nabla^{\bot}_{X} \xi}
\label{lift4}
\end{equation}

\begin{equation}
\hat{A}_{\hat{\xi}}J\eta=-(\hat{J\xi})^{\top}\qquad \hat{\nabla}^{\bot}_{J\eta}\hat{\xi}=(\hat{J\xi})^{\bot}
\label{lift5}
\end{equation}
for vector fields $X,\; Y$ tangent to $M$ and a vector field $\xi$ normal to $M$.

\subsection{Normal holonomy}
Given a submanifold $M$ of a (pseudo-)Riemannian manifold $N$, the normal holonomy group is the holonomy group associated to the normal connection $\nabla^{\bot}$ of $M$. Namely, given a piecewise differentiable curve $\gamma:I\to M$ such that $\gamma(0)=p$ and a normal vector $\xi_p\in \nu_p M$, one defines as usual the parallel displacement $\tau^{\bot}_{\gamma}(\xi_p)$ of $\xi_p$ along $\gamma$ with respect to the connection $\nabla^{\bot}$.

Set $\Omega_p(M)$ the set of piecewise differentiable loops of $M$ based at $p$ and by $\Omega^0_p(M)\subset\Omega_p(M)$ the set of null-homotopic piecewise differentiable loops of $M$ based at $p$. Then the normal holonomy group of $M$ at $p$ is defined as $$\Hol_p(M,\nabla^{\bot})=\{\tau^{\bot}_{\gamma}:\nu_p M\to \nu_pM\, :\, \gamma\in\Omega_p(M)\}\subset O(\nu_p M)$$
and the restricted normal holonomy group of $M$ at $p$ is the subgroup of $\Hol_p(M,\nabla^{\bot})$ defined as
$$\Hol_p^{0}(M,\nabla^{\bot})=\{\tau^{\bot}_{\gamma}:\nu_pM\to\nu_pM\, :\, \gamma\in\Omega_p^0(M)\}\subset SO(\nu_p M).$$
$\Hol_p^{0}(M,\nabla^{\bot})$ is the connected component of the identity of $\Hol_p(M,\nabla^{\bot})$.

\section{Coisotropic submanifolds: Proof of Theorem \ref{teoremagen}}

Observe that for the case $c=0$, Theorem \ref{teoremagen} is a direct consequence of the Normal Holonomy Theorem for real space forms \cite{Ol90}. Therefore we will prove it for $c\neq 0$.

\subsection{The strategy:}
The strategy will be the following. Consider a coisotropic submanifold $M$ of $\mathbb{S}_c$ and its pull-back $\hat{M}$ via the Hopf fibration $\pi:\overline{H}_c\to \mathbb{S}_c$. Then for each $p\in \hat{M}$, $d\pi_p$ defines an isometric isomorphism between $\nu_p\hat{M}$ and $\nu_{\pi(p)}M$ and conjugation by $d\pi_p$ defines an isomorphism between $SO(\nu_{\pi(p)}M)$ and $SO(\nu_p\hat{M})$. We will show that for any $p\in \hat{M}$:
\begin{enumerate}
\item The action of $\Hol^0_{\pi(p)}(M,\nabla^{\bot})$ on $\nu_{\pi(p)}M$ identifies, via conjugation with $d\pi_p$, with the action of $\Hol^0_p(\hat{M},\hat{\nabla}^{\bot})$ on $\nu_p\hat{M}$.
\item $\Hol^0_p(\hat{M},\hat{\nabla}^{\bot})$ acts on $\nu_p\hat{M}$ as the holonomy representation of a Riemannian symmetric space.
\end{enumerate}

We start with some technical results.
We will keep the notations introduced in section \ref{hopf}.

Fix some $p$ in $\hat{M}$ and set $x=\hat{\pi}(p)$. Let $\gamma(t)=e^{it}p$ be a vertical curve in $\hat{M}$ such that $\gamma(0)=p$. For $\xi\in \nu_x M$, let $\hat{\xi}(t)$ be the horizontal lift of $\xi$ to $\hat{M}$ at $\gamma(t)$. Then we have:
\begin{lema}
$M$ is a coisotropic submanifold if and only if $\hat{\xi}(t)$ is a $\hat{\nabla}^{\bot}$-parallel vector field along $\gamma(t)=e^{it}p$ for each $p\in \hat{M}$ and each $\xi\in T_{\hat{\pi}(p)}M$.
\label{lemagen}
\end{lema}
\begin{proof}
Fix $p$ in $M$ and let $\xi_p\in \nu_p\hat{M}$. Set $\xi:=d\hat{\pi}_p(\xi_p)$ and let $\hat{\xi}(t)$ be the normal vector field along $\gamma(t)=e^{it}p$ defined above. Observe first that $\hat{\xi}(t)=e^{it}\cdot\hat{\xi}_p$ (identifying each tangent space of $\overline{N}_c$ with a subspace of the ambient space).
So $$\nabla'_{\gamma'(t)}\hat{\xi}=\frac{d}{dt}_{|t}\left(e^{it}\cdot \hat{\xi}_p\right)=J\hat{\xi}(t).$$
Since $J\hat{\xi}(t)$ is the horizontal lift at $\gamma(t)$ of $J\xi$, we get that $\hat{\xi}$ is $\hat{\nabla}^{\bot}$-parallel in $\hat{M}$ for each $\xi\in \nu_{\hat{\pi}(p)}M$ and each $p\in \hat{M}$ if anf only if $J\nu_x M$ is tangent for each $x\in M$, that is, if and only if $M$ is a coisotropic submanifold.
\end{proof}

\begin{rem}
Let $J\eta$ be the  Hopf vector field of $\overline{N}_c$ and let $\{\varphi_t=e^{it}\}_{t\in\mathbb{R}}$ be its flow. Then  $\varphi_t$ is an isometry of $\overline{N}_c$ and  Lemma \ref{lemagen} can be stated as follows: $M$ is coisotropic submanifold if and only if $\varphi_t$ is a transvection with respect to the normal connection of $\hat{M}$, along $\varphi_t(p)$, for each $p\in \hat{M}$.
\end{rem}

\begin{lema}
Let $M$ be a coisotropic submanifold of a complex space form $\mathbb{S}_c$, with $c\neq 0$ and let $\hat{M}$ be its pullback via the Hopf fibration $\pi_c:\overline{N}_c\to \mathbb{S}_c$. Let $J\eta$ be the Hopf vector field. Fix $p\in \hat{M}$ and set $x=\pi(p)$.  Then
 \begin{enumerate}
 \item $\hat{R}^{\bot}_{J\eta,\hat{X}}=0$ for any horizontal vector $\hat{X}\in T_p M$;
 \item $[\hat{A}_{\hat{\xi}},\hat{A}_{\hat{\zeta}}]J\eta = 0$, for any $\hat{\xi},\ \hat{\zeta}\in \nu_p\hat{M}$;
 \item $A_{\xi}J\zeta=A_{\zeta}J\xi$ for any $\xi, \zeta\in \nu_x M$ (cf. \cite[Lemma 2.1]{YK})
     \end{enumerate}
\label{rh}
\label{Yano}
\end{lema}
\begin{proof}

We will prove first that statements ($1$), ($2$) and ($3$) are equivalent.
Let $\xi,\;\zeta\in\nu_x M$ and let $\hat{\xi}$, $\hat{\zeta}$ be the horizontal lifts of $\xi$ and $\zeta$ at $p$ respectively.

Since $\overline{N}_c$ is a real space form (a Lorentzian space form in the case of $c<0$) the Ricci equation gives
\begin{equation}
\langle \hat{R}^{\bot}_{J\eta,\hat{X}}\hat{\xi},\hat{\zeta}\rangle=\langle[\hat{A}_{\hat{\xi}},\hat{A}_{\hat{\zeta}}]J\eta,\hat{X}\rangle.
\label{ricci2}
\end{equation}
for any horizontal tangent vector $\hat{X}\in T_p\hat{M}$. This proves the equivalence between (1) and (2).

On the other hand, from equations (\ref{lift4}) and (\ref{lift5}) we have

\begin{eqnarray*}
[\hat{A}_{\hat{\xi}},\hat{A}_{\hat{\zeta}}]J\eta&=&-\hat{A}_{\hat{\xi}}(\widehat{J\zeta})+\hat{A}_{\hat{\zeta}}(\widehat{J\xi})\\
 &=& -\widehat{A_{\xi}J\zeta}+\langle J\xi,J\zeta\rangle J\eta+\widehat{A_{\zeta}J\xi}-\langle J\xi,J\zeta\rangle J\eta\\
 &=&\widehat{A_{\xi}J\zeta-A_{\zeta}J\xi}
\end{eqnarray*}
This proves the equivalence between (2) and (3).

Let us prove (3). Let $Y\in T_x M$. Then
\begin{eqnarray*}
\meti A_{\xi} J\zeta, Y\metd &=& \meti \alpha(J\zeta, Y),\xi\metd =\meti \overline{\nabla}_Y J\zeta,\xi\metd =\meti J\overline{\nabla}_Y \zeta,\xi\metd\\
&=& \meti A_{\zeta}Y,J\xi\metd= \meti A_{\zeta}J\xi, Y\metd.
\end{eqnarray*}
Since $Y$ is arbitrary, this concludes the proof.
\end{proof}

The following is an immediate consequence of Lemma \ref{rh} and a result in \cite[Appendix]{Ol93}.
\begin{cor}
Let $\hat{M}$ be the pull-back to $\overline{N}_c$ of a coisotropic submanifold $M$ of a complex space form $\mathbb{S}_c$, $c\neq 0$ via the Hopf fibration $\pi$. Then for any piecewise-differentiable curve $\sigma:I\to \hat{M}$ there exist a horizontal curve $\sigma_0$ and a vertical curve $\gamma$ (with respect to $\pi$) such that $$\hat{\tau}^{\bot}_{\sigma}=\hat{\tau}^{\bot}_{\gamma}\circ \hat{\tau}^{\bot}_{\sigma_0},$$
where $\hat{\tau}^{\bot}$ denotes the $\hat{\nabla}^{\bot}$ parallel displacement on $\hat{M}$.
\label{corrh}
\end{cor}

\begin{prop} Let $M$ be a coisotropic submanifold of a complex space form $\mathbb{S}_c$, with $c\neq 0$ and let $\hat{M}$ be its pullback via the Hopf fibration $\pi_c:\overline{N}_c\to \mathbb{S}_c$.  Fix $p\in \hat{M}$ and set $x=\pi(p)$.  Then the action of $\Hol^0_{x}(M,\nabla^{\bot})$ on $\nu_{x}M$ identifies, via conjugation with $d\pi_p$, with the action of $\Hol^0_p(\hat{M},\hat{\nabla}^{\bot})$ on $\nu_p\hat{M}$.
\label{identificacion}
\end{prop}
\begin{proof}

Fix $p\in\hat{M}$ and let $x=\pi(p)\in M$. We can consider the restricted normal holonomy group of $M$ at $x$ acting on $\nu_p\hat{M}$ via $d\pi_p$ in the following way.

If $\tau^{\bot}\in \Hol^{0}_x(M,\nabla^{\bot})$ and $\hat{\xi}\in\nu_p \hat{M}$, then $$\tau^{\bot}\cdot \hat{\xi}=(d\pi_{p|\nu_p\hat{M}})^{-1}\circ\tau^{\bot}\circ d\pi_p(\hat{\xi}).$$

If $c(t)$ is a loop in $M$ based at $x$, then its horizontal lift $\hat{c}(t)$ at $p$ is a curve in $\hat{M}$ such that $\pi(\hat{c}(1))=\pi(p)=x$. There is a unique vertical curve $\delta(t)=e^{i\theta t}\hat{c}(1)$ in $\hat{M}$ joining $\hat{c}(1)$ and $\hat{c}(0)=p$, for some fixed real number $\theta$. Consider the loop $\sigma(t)$ based at $p$ obtained by moving along $\hat{c}$ from $p$ to $\hat{c}(1)$  and then along $\delta$ from $\hat{c}(1)$ back to $p$. Then if $\tau^{\bot}$ is the $\nabla^{\bot}$-parallel displacement in $M$ along $c$,
from  (\ref{lift4}) and Lemma \ref{lemagen} we obtain that the parallel displacement of any normal vector $\xi_p\in \nu_p\hat{M}$ along $\sigma$ is actually $\tau^{\bot}\cdot  \xi_p$.

Conversely, if $\sigma$ is a loop in $\hat{M}$ based at $p$, then by Corollary \ref{corrh}, there exist a vertical curve $\gamma$ and a horizontal curve $\sigma_0$ starting at $p$ such that $\hat{\tau}^{\bot}_{\sigma}=\hat{\tau}^{\bot}_{\gamma}\circ\hat{\tau}^{\bot}_{\sigma_0}$. Now, if $\tau^{\bot}$ is the $\nabla^{\bot}$-parallel displacement in $M$ along the loop $\pi(\sigma_0)$, then from equation (\ref{lift4}) and Lemma \ref{lemagen} it is easy to see that for any $\xi_p\in\nu_p\hat{M}$,  $\hat{\tau}^{\bot}_{\gamma}\circ\hat{\tau}^{\bot}_{\sigma_0}(\xi_p)=\tau^{\bot}\cdot\xi_p$.

This shows that the action of $\Hol^{0}_x(M,\nabla^{\bot})$ on $\nu_x M$ is the same, via conjugation with $d\hat{\pi}_p$, as the action of $\Hol^{0}_p(\hat{M},\hat{\nabla}^{\bot})$ on $\nu_p \hat{M}$.
\end{proof}

\subsection{Proof of Theorem \ref{teoremagen} when $c>0$}

From Proposition \ref{identificacion} it follows that:

1) The action of $\Hol^0_{\pi(p)}(M,\nabla^{\bot})$ on $\nu_{\pi(p)}M$ identifies, via conjugation with $d\pi_p$, with the action of $\Hol^0_p(\hat{M},\hat{\nabla}^{\bot})$ on $\nu_p\hat{M}$.

Since $\hat{M}$ is a submanifold of a sphere, the Normal Holonomy Theorem \cite{Ol90} implies that:

2) $\Hol^0_p(\hat{M},\hat{\nabla}^{\bot})$ acts on $\nu_p\hat{M}$ as the holonomy representation of a Riemannian symmetric space.

This proves Theorem \ref{teoremagen} when $c>0$. \hfill $\Box$

\subsection{Proof of Theorem \ref{teoremagen} when $c<0$}

From Proposition \ref{identificacion} it follows that:

1) The action of $\Hol^0_{\pi(p)}(M,\nabla^{\bot})$ on $\nu_{\pi(p)}M$ identifies, via conjugation with $d\pi_p$, with the action of $\Hol^0_p(\hat{M},\hat{\nabla}^{\bot})$ on $\nu_p\hat{M}$.

Since the pull-back $\hat{M}$ is a Lorentzian submanifold of the anti-De-Sitter space, Olmos' normal holonomy theorem \cite{Ol90} can not be used directly. Actually, Olmos' normal holonomy theorem is not true for an arbitrary Lorentzian submanifold of the anti-De-Sitter space.
However, Olmos' proof can be adapted to our case i.e. when $\hat{M}$ is the pull-back of a coisotropic submanifold of the complex hyperbolic space.

Theorem \ref{teoremagen} for $c<0$ is a consequence of Proposition \ref{identificacion} and the following result.


\begin{prop}
\label{normalholonomy}
Let $M$ be a coisotropic submanifold of the complex hyperbolic space $\mathbb{C}H^n$ and let $\pi:H^{n+1}_1\to \mathbb{C}H^n$ be the Hopf fibration. Let $\hat{M}\subset H^{n+1}_1$ be the pull-back of $M$.  Let $p\in \hat{M}$ and let $\Hol^0_p(\hat{M},\hat{\nabla}^{\perp})$ be the restricted normal holonomy group at $p$.

Then $\Hol^0_p(\hat{M},\hat{\nabla}^{\perp})$ is compact, there exists a unique (up to order) orthogonal decomposition $\nu_p\hat{M}=V_0\oplus \cdots\oplus V_k$ of the normal space $\nu_p \hat{M}$ into $\Hol^0_p(\hat{M},\hat{\nabla}^{\perp})$-invariant subspaces and there exist normal subgroups $\Phi_0,\cdots,\Phi_k$ of $\Hol^0_p(\hat{M},\hat{\nabla}^{\perp})$ such that
\begin{itemize}
\item[i)] $\Hol^0_p(\hat{M},\hat{\nabla}^{\perp})=\Phi_0\times \cdots \times \Phi_k$ (direct product);
\item[ii)] $\Phi_i$ acts trivially on $V_j$ if $i\neq j$;
\item[iii)] $\Phi_0=\{1\}$ and if $i\geq 1$, $\Phi_i$ acts irreducibly on $V_i$ as the isotropy representation of an irreducible Riemannian symmetric space.
\end{itemize}
\end{prop}

\begin{proof}
The key object in Olmos' proof is the algebraic curvature tensor $\mathcal{R}^{\bot}$ on $\nu \hat{M}$ {\bf with non positive sectional curvature} and that carries the same geometric information as the normal curvature tensor $\hat{R}^{\perp}$ of $\hat{M}$.

Following \cite{Ol90} we introduce the algebraic curvature tensor $\mathcal{R}^{\bot}$ on $\nu \hat{M}$ by the formula:
\[\meti \crn(\xi_1,\xi_2)\xi_3,\xi_4\metd := -\frac{1}{2}tr([\hat{A}_{\xi_1},\hat{A}_{\xi_2}]\circ[\hat{A}_{\xi_3},\hat{A}_{\xi_4}]) \]

So in the same way as in \cite{Ol90}, one can prove

\begin{itemize}
\item[(i)] $\crn(\xi_1,\xi_2)=-\crn(\xi_2,\xi_1)$;
\item[(ii)] $\meti \crn(\xi_1,\xi_2)\xi_3,\xi_4\metd=-\meti \xi_3,\crn (\xi_1,\xi_2)\xi_4\metd$;

\item[(iii)] $\meti \crn(\xi_1,\xi_2)\xi_3,\xi_4\metd= \meti \crn(\xi_3,\xi_4)\xi_1,\xi_2\metd$

\item[(iv)] $\crn(\xi_1,\xi_2)\xi_3+\crn(\xi_2,\xi_3)\xi_1+\crn(\xi_3,\xi_1)\xi_2=0$.

\item[(v)] $\displaystyle
\mathrm{Im}(\crn_p)=\mathrm{Im}(\hat{R}_p^{\bot}).$
\end{itemize}

Now we compute the sectional curvature $\langle \mathcal{R}^{\bot}(\xi,\zeta)\zeta,\xi \rangle$.
Choose an orthonormal basis $\{e_0,e_1,\cdots,e_{k}\}$ of $T_p\hat{M}$ such that $e_0=J\eta$ is the Hopf vector and therefore $e_1,\cdots, e_{k}$ are horizontal vectors.

So we have
\begin{eqnarray*}
\langle \mathcal{R}^{\bot}(\xi,\zeta)\zeta,\xi \rangle &=&\frac{1}{2}tr([\hat{A}_{\xi},\hat{A}_{\zeta}]^2)\\
&=&-\frac{1}{2}\meti [\hat{A}_{\xi},\hat{A}_{\zeta}]^2J\eta,J\eta\metd+\frac{1}{2}\sum_{i=1}^{k}\meti [\hat{A}_{\xi},\hat{A}_{\zeta}]^2e_i,e_i\metd\\
&=&-\frac{1}{2}\sum_{i=1}^{k}\meti [\hat{A}_{\xi},\hat{A}_{\zeta}]e_i,[\hat{A}_{\xi},\hat{A}_{\zeta}]e_i\metd
\leq 0
\end{eqnarray*}
since by Lemma \ref{rh} $[\hat{A}_{\xi},\hat{A}_{\zeta}]J\eta=0$ hence the vectors $[\hat{A}_{\xi},\hat{A}_{\zeta}]e_i$ are horizontal.

Observe also that
\begin{equation}
\langle \mathcal{R}^{\bot}(\xi,\zeta)\zeta,\xi \rangle = 0 \text{ if and only if }[\hat{A}_{\zeta},\hat{A}_{\xi}]=0
\label{zero}
\end{equation}

Now the proof follows as in \cite{Ol90}.
\end{proof}

\section{Lagrangian submanifolds: Proof of Corollary \ref{cor:RicciFlat}}

Since Lagrangian submanifolds are in particular coisotropic submanifolds we get the following result.

\begin{thm} \label{Lagrangian} Let $M$ be a Lagrangian submanifold a complex space form and let $\Hol^0_p(M,\nabla^{\perp})$ be the restricted holonomy group of the normal connection. Then $\Hol^0_p(M,\nabla^{\perp})$ acts on the normal space $\nu_pM$ as the holonomy representation of a Riemannian symmetric space.
\label{lagrangian}
\end{thm}

We give now the proof of Corollary \ref{cor:RicciFlat}.

\begin{proof} Since $M$ is a Lagrangian submanifold,  the complex structure $J$ defines  an isomorphism between the tangent space $T_pM$ and the normal space $\nu_pM$. Let $\xi(t)$ be a $\nabla^{\perp}$-parallel normal field along a loop $\gamma(t)$ based at $p$. Then $J\xi(t)$ is a tangent field along $\gamma(t)$ which is parallel with respect to the Levi-Civita connection $\nabla$ of $M$. Indeed, \[ \overline{\nabla}_{\gamma'(t)} J \xi(t) = J \left( \nabla^{\perp}_{\gamma'(t)}\xi(t) - A_{\xi(t)}(\gamma'(t)) \right) = - J A_{\xi(t)}(\gamma'(t)), \]
which shows that $\overline{\nabla}_{\gamma'(t)} J \xi(t)$ is normal to $M$, and hence  $\nabla_{\gamma'(t)} J \xi(t)=0$.

In a similar way, if $X(t)$ is a vector field of $M$ along $\gamma$, parallel with respect to the Levi-Civita connection of $M$, we get that $JX(t)$ is a
$\nabla^{\perp}$-parallel vector field along $\gamma$.

Then the isomorphism $J$ is an intertwiner isomorphism between the normal and the tangent holonomy groups. Hence, by the above theorem, the tangent holonomy group acts as an s-representation.

So the tangent holonomy group is the holonomy group of a Riemannian symmetric space. Since $M$ is Ricci flat, either  $\Hol^0(M,\nabla^{\perp})=SO(TM)$ or $M$ is flat. \end{proof}

\section{Totally real submanifolds: some results about their normal holonomy group}
\label{Section:TotallyReal}

Along this section we keep the notations of Section \ref{hopf}.

\subsection{An example}

Let us give an example showing that the strategy we followed for coisotropic submanifolds does not work for totally real submanifolds. Namely, the normal holonomy group of the pull-back $\hat{M}$ can be different from the normal holonomy group of $M$.\\

Let $M \subset \mathbb{C}P^n$, $n>1$ be a curve of the complex projective space. Then the normal holonomy group of $M$ is of course trivial.

To compute the normal curvature tensor of $\hat{M}$ we need to compute the shape operators of $\hat{M}$. Let $T$ be a unit vector field tangent to $M$ and denote by $\hat{T}$ its horizontal lift to $\hat{M}$. Then $\{ J\eta, \hat{T} \}$ is an orthogonal frame of $T\hat{M}$.
Observe that the normal bundle of $M$ splits as $\nu M = \mathbb{R} JT \oplus (\mathbb{R} JT)^{\bot}$ where $J(\mathbb{R} JT)^{\bot} = (\mathbb{R} JT)^{\bot}$.
Then the normal bundle of the pull-back $\hat{M}$ splits as
\[ \nu\hat{M} = \mathbb{R }J\hat{T} \oplus  (\mathbb{R }J\hat{T})^{\bot}\]
where $J(\mathbb{R }J\hat{T})^{\bot} = (\mathbb{R }J\hat{T})^{\bot}$.
Let $\xi$ be a section of  $(\mathbb{R} JT)^{\bot}$ and consider the section $\hat{\xi}$ of $\nu(\hat{M})$.
Then, by equations (\ref{lift4}) and (\ref{lift5}) the shape operator  $\hat{A}_{\hat{\xi}}$ of $\hat{M}$ in direction $\hat{\xi}$ is given in the frame $\{ J\eta, \hat{T} \}$ by the following $2 \times 2$ matrix: \[ \hat{A}_{\hat{\xi}} =  \left(
                                           \begin{array}{cc}
                                             0 & 0 \\
                                             0 & \langle A_{\xi}(T), T \rangle \\
                                           \end{array}
                                         \right) \, \, .
\]
The shape operator  $\hat{A}_{J\hat{T}}$ of $\hat{M}$ is given in the frame $\{ J\eta, \hat{T} \}$ by the following $2 \times 2$ matrix: \[ \hat{A}_{J\hat{T}} =  \left(
                                           \begin{array}{cc}
                                             0 & 1 \\
                                             1 & \langle A_{J T}(T), T \rangle \\
                                           \end{array}
                                         \right) \, \, .
\]

So we have the following proposition.
\begin{prop} Let $M \subset \mathbb{C}P^n$ be a curve of the complex projective space. Then the pull-back $\hat{M}$ has flat normal bundle if and only if \[ \langle A_{\xi}(T), T \rangle  = 0 \]
for all $\xi \in \nu(M)$ such that $\langle \xi, JT \rangle = 0$. Equivalently, $\hat{M}$ has flat normal bundle if and only if the curve $M$ is a so called holomorphic circle \cite[page 8, Definition]{AMU} (also called K\"ahler-Frenet curve \cite[Introduction]{MT}). Namely, \[ \overline{\nabla}_T T = \kappa J T \, \,  \]
where $\kappa$ is a smooth function on $M$.
\end{prop}
\begin{proof} The pull-back $\hat{M}$ is a submanifold of a sphere. Then by the Ricci equation, $M$ has flat normal bundle if and only if all shape operators commute. Then the conclusion follows from
\[ [\hat{A}_{\hat{\xi}}, \hat{A}_{J\hat{T}}] = \left(
                                                 \begin{array}{cc}
                                                   0 & -\langle A_{\xi}(T), T \rangle \\
                                                \langle A_{\xi}(T), T \rangle & 0 \\
                                                 \end{array}
                                               \right)
\]
\end{proof}

\begin{rem} Interesting examples of holomorphic circles are the so called magnetic geodesics (cf. \cite[Introduction]{GS}).
\end{rem}

Then if $M \subset \mathbb{C}P^n$ is not a holomorphic circle we get a submanifold whose normal holonomy group is different from the normal holonomy group of its pull-back $\hat{M}$. Indeed, by the above proposition the normal bundle of $\hat{M}$ is not flat whilst the normal bundle of $M$ is. Moreover, by \cite[Exercise 4.6.16, page 136]{BCO} we get the following proposition.
\begin{prop} Let $M$ be a full curve of $\mathbb{C}P^n$ which is not an holomorphic circle and let $\hat{M}$ be its pull-back. Then the normal holonomy group $\Hol_p(\hat{M}, \hat{\nabla}^{\perp})$ acts transitively on the unit sphere of the normal space.
\end{prop}

\subsection{Injection of the normal holonomy group}

At the light of the above example one can not expect to identify the normal holonomy of $M$ with that of its pull-back $\hat{M}$.
However, we will show that the holonomy group of $M$ injects into the normal holonomy group of $\hat{M}$.
We will need the following lemma.

\begin{lema}
Let $M$ be a  submanifold of a space form $\mathbb{S}_c$, with $c\neq 0$, and let $\hat{M}$ be its pullback to $\overline{N}_c$. Then $M$ is totally real if and only if the horizontal distribution $\hat{H}$ is a parallel distribution of $\hat{M}$.
\end{lema}

\begin{proof} The only if part was proved in \cite[Lemma 1.1]{NT}. For convenience of the reader we give here a proof.
From equation (\ref{lift3}) it is immediate to see that if $M$ is totally real then $\hat{H}$ is an autoparallel distribution and therefore parallel since the Hopf vector field is geodesic.

On the other hand, if $H$ is parallel, equation (\ref{lift3}) implies $\meti X, JY\metd=0$ for every $X,Y\in \mathfrak{X}(M)$ and so $M$ is totally real.
\end{proof}

\begin{thm}
Let $M$ be a totally real submanifold of a complex space form $\mathbb{S}_c$ with $c\neq 0$. Let $\hat{M}$ be its pullback to $\overline{N}_c$. Then the normal holonomy group $\Hol_p(M,\nabla^{\bot})$ is a subgroup of $\Hol_{\hat{p}}(\hat{M},\hat{\nabla}^{\bot})$, where $\hat{p}$ is any point of $\hat{M}$ such that $\pi_c(\hat{p})=p$.
\end{thm}

\begin{proof}
Let $\sigma$ be a loop in $M$ based at $p$ and let $\hat{\sigma}$ be its horizontal lift to $\hat{M}$ at $\hat{p}$. Since $\hat{H}$ is an integrable distribution, one gets that $\hat{\sigma}$ is also a loop in $\hat{M}$ based at $\hat{p}$.

Moreover $d\hat{\pi}$ defines an isometry between the normal spaces of $\hat{M}$ and $M$, which from equation (\ref{lift4})  preserves parallel transport along horizontal curves. This implies that the map $\Phi:\Hol_p(M,\nabla^{\bot})\to \Hol_{\hat{p}}(\hat{M},\hat{\nabla}^{\bot})$ given by $\Phi(\tau^{\bot}_\sigma)=\tau^{\bot}_{\hat{\sigma}}$ is an injective homomorphism.
\end{proof}

\subsection{Reduction of codimension}

\begin{thm} \label{reduccion}
Let $M$ be a totally real submanifod of a complex space form $\mathbb{S}^{m}_c$.
\begin{enumerate}
\item There exists a totally geodesic complex submanifold $\mathbb{S}^{n}_c$ of $\mathbb{S}^{m}_c$ such that $M\subset \mathbb{S}^{n}_c$ if and only if there exists a $\nabla^{\bot}$-parallel sub-bundle $W_0$ of $\nu M$ such that $TM\oplus W_0$ is $J$-invariant.

\noindent
    If in particular $JM$ is $\nabla^{\bot}$-parallel (i.e. $W_0=JM$), then $M$ is a Lagrangian submanifold of $\mathbb{S}^{n}_c$ (cf. \cite{CHL}).
\item There exists a totally geodesic totally real submanifold $N$ of $\mathbb{S}^{m}_c$ such that $M\subset N$ if and only if there exists a $\nabla^{\bot}$-parallel subbundle $W_0$ of $\nu M$ such that the first normal space $N^{1}$ of $M$ is contained in $W_0$ and $W_0\;\bot\;J(TM\oplus W_0)$.
\end{enumerate}
\end{thm}

\begin{proof}
Let $M$ be a totally real submanifold of a complex space form $\mathbb{S}_c$.  By the result in \cite{DV} we know that if the first normal space $\mathrm{N}^1=\alpha(TM\times TM)$ is contained in a $\nabla^{\bot}$-parallel sub-bundle $W$ of $\nu M$ such that $\mathrm{V}:=TM\oplus W$ is $\overline{R}^{c}$-invariant, then $M$ is contained in a totally geodesic submanifold $N$ of $\mathbb{S}_c$ of dimension equal to $\mathrm{rank}(\mathrm{V})$.
We are going to show that this is indeed the case in both items.\\

Case (1).
Assume that there is a $\nabla^{\bot}$-parallel sub-bundle $W_0$ of $\nu M$ such that $\mathrm{V}:=TM\oplus W_0$ is $J$-invariant. From equation (\ref{curvatura}) one can easily see that $\mathrm{V}$ is $\overline{R}^c$-invariant.

Since $\mathrm{V}$ is $J$-invariant and $M$ is totally real, one has that $$W_0=J(TM)\oplus W_1$$ and $W_1$ is $J$-invariant. Let $W_2=(W_0)^{\bot} \subset \nu M$. So the normal bundle of $M$ decomposes as $\nu M=J(TM)\oplus W_1\oplus W_2$. Given two tangent vectors $X,Y$ to $M$, set $\alpha(X,Y)=\xi+\xi_1+\xi_2$, where $\xi \in J(TM)$, $\xi_1 \in W_1$ and $\xi_2 \in W_2$.
Then one has
 \begin{equation}
 \overline{\nabla}_X JY=J\overline{\nabla}_X Y=J \nabla_X Y+ J \xi + J\xi_1+J\xi_2 \, \, .
 \label{nabla1}
 \end{equation}
 On the other hand, $\overline{\nabla}_X JY=-A_{JY}X+\nabla^{\bot}_X JY$. Comparing with the normal part in (\ref{nabla1}) we get  \[ \nabla^{\bot}_X JY = J \nabla_X Y+ J\xi_1+J\xi_2 \, \, .\]
 Since $J(TM)\subset W_0$ and $W_0$ is $\nabla^{\bot}$-parallel we get $$J\xi_2= \nabla^{\bot}_X JY-J\xi_1 - J\nabla_X Y\in W_0$$
Since $W_2$ is also $J$-invariant, we get $J\xi_2=0$ and so $\alpha(X,Y)\in W_0$.

Therefore the first normal space is contained in $W_0$ and $M$ is contained in a totally geodesic submanifold $N$ of $\mathbb{S}_c$ whose tangent bundle (along $M$) is $TM\oplus W_0=\mathrm{V}$. Since $V$ is $J$-invariant it follows that $N$ is a complex totally geodesic submanifold hence a complex space form $\mathbb{S}^n_c$.

The converse is immediate, taking $W_0=\nu M\cap T\mathbb{S}_c^{n}$.\\

Case (2).
Assume now that  there exists a $\nabla^{\bot}$-parallel sub-bundle $W_0$ of $\nu M$ such that  $N_1\subset W_0$ and $W_0\;\bot\; J(TM\oplus W_0)$. Then it is not difficult to see, from equation (\ref{curvatura}) that $\mathrm{V}:= TM\oplus W_0$ is $\overline{R}^c$-invariant. Since $W_0$ contains the first normal space of $M$, there is a totally geodesic submanifold $N$ of $\mathbb{S}^{n}$ containing $M$ whose tangent bundle along $M$ is $TM\oplus W_0$. This implies that $N$ is totally real.
Conversely, assume that $M$ is contained in a  totally geodesic totally real submanifold $N$ of $\mathbb{S}_c^{m}$. Then $W_0:= \nu M\cap TN$ satisfies the conditions of the statement.
\end{proof}

\begin{remark}
\label{remarkcompleja}
Let $M$ be a non-full totally real submanifold of a complex space form $\mathbb{S}_c^{m}$ contained in a complex space form $\mathbb{S}_c^{n}\subset \mathbb{S}_c^{m}$. Then $\left(\nu \mathbb{S}_c^{m}\right)_{|M}$ is contained in $\nu_0(M)$, the largest parallel and flat sub-bundle of $\nu M$.

In fact $\left(\nu \mathbb{S}_c^{n}\right)_{|M}$ is $\nabla^{\bot}$-parallel. So, to prove the inclusion $\left(\nu \mathbb{S}_c^{m}\right)_{|M}\subset\nu_0(M)$, it is enough to see that $R^{\bot}_{X,Y}\xi=0$ for every $\xi\in \left(\nu \mathbb{S}_c^{m}\right)_{|M}$. To see this, observe that the first normal space of $M$ is contained in  $T\mathbb{S}_c^{n}$ and so if $\xi,\; \zeta\in \left(\nu \mathbb{S}_c^{m}\right)_{|M}$ then $A_{\xi}=A_{\zeta}=0$. Hence from the Ricci equation (\ref{ricci}) we have $$\meti R^{\bot}_{X,Y}\xi,\zeta\metd=\meti \overline{R}^{c}_{X,Y}\xi,\zeta\metd=0$$
for every $X,\; Y\in \mathfrak{X}(M)$.
\end{remark}

From Theorem \ref{lagrangian}, Theorem \ref{reduccion} and Remark \ref{remarkcompleja} one immediately gets the following

\begin{cor}
Let $M$ be a totally real submanifold of a complex space form $\mathbb{S}_c^{n}$. If $J(TM)$ is a $\nabla^{\bot}$-parallel sub-bundle of $\nu M$, then the restricted normal holonomy group $\Hol^0(M,\nabla^{\bot})$ acts on each normal space as the holonomy representation of a symmetric space  (i.e. a flat factor plus an s-representation).
\end{cor}

Now we compute the normal holonomy group $\Hol(M,\nabla^{\perp})$ of a totally real submanifold
$M\subset N\subset \mathbb{S}_c^{m}$, where $N$ is a totally geodesic totally real submanifold of $\mathbb{S}_c^{m}$, i.e., $N$ is a real projective space in the case $c>0$ and $N$ is a real hyperbolic space in the case $c <0$.

\begin{thm}
\label{reduccionreal}
Let $M$ be a totally real submanifold of a complex space form $\mathbb{S}_c^{m}$ contained in a totally real, totally geodesic submanifold $N$ of $\mathbb{S}_c^{m}$, with $dim(M)\geq 2$.   The normal bundle $\nu M$ decomposes as the sum of the $\nabla^{\bot}$-parallel subbundles $$\nu M=\nu_N M\oplus \nu N_{|M},$$
where $\nu_N M$ is the normal bundle of $M$ as a submanifold of $N$. Then

\begin{enumerate}

\item  The restricted normal holonomy group acts on $\nu_{N}M $ as the holonomy representation of a symmetric space.

\item The parallel subbundle $\nu N_{|M}$ splits as
\[ \nu N_{|M} = W \oplus W^{\perp}\]
where $W$ is the smallest $\nabla^{\perp}$-parallel subbundle containing $J (TM)$. The group $\Hol(M,\nabla^{\perp})$ acts trivially on $W^{\perp}$ and it is the full orthogonal group on $W$ i.e. $\mathfrak{hol}(M,\nabla^{\perp})_{|W} = \mathfrak{so}(W)$.

\item $W = J(TM)$ if and only if $M$ is a totally geodesic submanifold of $\mathbb{S}_c^{m}$.

\end{enumerate}

\end{thm}

\begin{proof} The parallel splitting $\nu(M) = \nu_{N }M \oplus \nu N_{|M}$ follows from the fact that $N$ is a totally geodesic submanifold of $M$. Since $N$ is also totally real, it is a real space form and so $\Hol^0(M,\nabla^{\perp})$ acts on $\nu_{N}M $ as the holonomy representation of a symmetric space as follows from Olmos' Theorem \cite{Ol90}. This proves (1).\\

For (2), let $E_1 := \nabla^{\perp} J(TM)$ be the subbundle of $\nu(N)_{|M}$ obtained by taking derivatives of sections of $J(TM)$. By taking further derivatives we get the subbundles $E_j := \underbrace{\nabla^{\perp} \cdots \nabla^{\perp}}_{j-times} J(TM) $. So  the smallest $\nabla^{\perp}$-parallel subbundle of $\nu N_{|M}$  containing $J (TM)$ is \[ W = J(TM) + E_1 + E_2 + E_3 + \cdots \, . \]
To prove that $\mathfrak{hol}(M,\nabla^{\perp})_{|W} = \mathfrak{so}(W)$ notice that the curvature tensor of $\nabla^{\perp}$ on $\nu N_{|M}$  is $$R^{\perp}_{X,Y} \xi = \frac{1}{4}c (JX \wedge JY ) (\xi) .$$
This immediately implies that $W^{\perp}$ is flat i.e. the action of $\Hol(M,\nabla^{\perp})$ is trivial and that
$\Lambda^2 J(TM) \subset \mathfrak{hol}(M,\nabla^{\perp})_{|W}$.
As it is well-known the covariant derivatives $\nabla^{j} R^{\perp}$ also belong to the holonomy algebra $\mathfrak{hol}(M,\nabla^{\perp})_{|W}$.
As consequence we get that $J(TM) \wedge E_i$ and $E_i \wedge E_j$ are both contained in $\mathfrak{hol}(M,\nabla^{\perp})_{|W}$. Thus, $\Lambda^2 W$ is contained in the Lie algebra $\mathfrak{hol}(M,\nabla^{\perp})_{|W}$ which proves that $\mathfrak{hol}(M,\nabla^{\perp})_{|W} = \mathfrak{so}(W)$. This proves (2).\\

To prove (3) observe that $W=J(TM)$ if and only if $J(TM)$ is $\nabla^{\bot}$-parallel.
Take $X,Y\in \mathfrak{X}(M)$. Then comparing the normal parts of $\overline{\nabla}_X JY=J\overline{\nabla}_X Y$ one gets $$J\nabla_X Y+J\alpha(X,Y)=\nabla^{\bot}_X JY$$
since $\alpha(TM\times TM)\cap J(TM)=\{0\}$. Therefore $J(TM)$ is parallel if and only if $J\alpha(X,Y)=0$, i.e. $M$ is totally geodesic. This proves (3) and completes the proof of the theorem.
\end{proof}

One of the consequences of Olmos' holonomy theorem is the compactness of the restricted normal holonomy group of a submanifold of a real space form. In general there are no reasons to expect the compactness of the normal holonomy group for submanifolds of a Riemannian space even in the case of submanifolds of symmetric spaces (e.g. \cite[Theorem 10, (b,i)]{AD}).

For a totally real submanifold contained in a totally real totally geodesic submanifold of a complex space form $\mathbb{S}_c^{m}$ the following theorem shows that the normal holonomy group is indeed compact.

We will need the following lemma which is a standard consequence of \cite[Proposition 6.6., page 122]{Helgason}.
\begin{lema}\label{Helgason} Let $K$ be a compact connected Lie group and let $N \lhd K$ be a normal subgroup of $K$.
If the center of $K$ is contained in $N$, then $N$ is closed, and  hence compact in $K$.
\end{lema}

\begin{thm}
\label{thm:compactness}
Let $M$ be a totally real submanifold of a complex space form $\mathbb{S}_c^{m}$ contained in a totally real, totally geodesic submanifold $N$ of $\mathbb{S}_c^{m}$. Then the restricted normal holonomy group $\Hol^{0}_p(M,\nabla^{\bot})$ at $p \in M$ is compact.
\end{thm}

\begin{proof}
According to the decomposition $\nu M=\nu_N M\oplus W\oplus W^{\bot}$, any element $\tau$ of  $\Hol^{0}_p(M,\nabla^{\bot})$ has the block diagonal form
$$\tau=\left( \begin{array}{ccc}
A & 0 & 0\\
0 & B & 0\\
0 & 0 & 1 \\
\end{array}\right) \, \, .$$

By item $(2)$ of Theorem \ref{reduccionreal} the map  $\phi: \Hol^{0}_p(M,\nabla^{\bot}) \to SO(W)$ defined by $\phi(\tau) = B$ gives rise to the following short exact sequence of groups
\begin{equation}\label{shortSequence} 0 \rightarrow \mathrm{Ker}(\phi) \to  \Hol^{0}_p(M,\nabla^{\bot}) \to SO(W) \to 0 \end{equation}

Then to show that $\Hol^{0}_p(M,\nabla^{\bot})$ is compact it is enough to show that $\mathrm{Ker}(\phi)$ is compact.
By definition we have that $\tau \in \mathrm{Ker}(\phi)$ if and only if \[ \tau =\left( \begin{array}{ccc}
A & 0 & 0\\
0 & 1 & 0\\
0 & 0 & 1 \\
\end{array}\right) \]
and so we have an injective map $\psi: Ker(\phi) \to \Hol^{0}_p(M \subset N,\nabla^{\bot})$, where $\Hol^{0}_p(M \subset N,\nabla^{\bot})$ is the normal holonomy group of $M$ regarded as a submanifold of $N$.\\

{\bf Claim 1:} The image $\psi(Ker(\phi))$ is a normal subgroup of $\Hol^{0}_p(M \subset N,\nabla^{\bot})$.\\

 Indeed, any element $\mathbf{x}$ of $\Hol^{0}_p(M \subset N,\nabla^{\bot})$ is determined by a null-homotopic loop $\gamma$ in $M$ based at $p$. Then the parallel transport $\tau_{\gamma} \in \Hol^{0}_p(M,\nabla^{\bot})$ has the matrix
\[ \tau_{\gamma} =\left( \begin{array}{ccc}
\mathbf{x} & 0 & 0\\
0 & B & 0\\
0 & 0 & 1 \\
\end{array}\right) \]
Then that $\mathrm{Ker}(\phi)$ is a normal subgroup of $\Hol^{0}_p(M,\nabla^{\bot})$ follows from a direct computation with the diagonal block decomposition.\\

{\bf Claim 2:} The center of $\Hol^{0}_p(M \subset N,\nabla^{\bot})$ is contained in $\psi(Ker(\phi))$.\\

Indeed, the short sequence (\ref{shortSequence}) induces a morphism \[ \rho : SO(W) \to \Hol^{0}_p(M \subset N,\nabla^{\bot})/ \psi(\mathrm{Ker}(\phi)) \, .\]Observe that if $\mathrm{dim}(W) > 2$ then $SO(W)$ is a semisimple Lie group hence the center of $\Hol^{0}_p(M \subset N,\nabla^{\bot})$ must be contained in $\psi(\mathrm{Ker}(\phi))$.
If $\mathrm{dim}(W) =1$ then $M$ is a curve so the claim is trivial. If  $\mathrm{dim}(W) = 2$ then $M$ is a surface with parallel $JTM$. So part (3) of Theorem \ref{reduccionreal} implies that $\Hol^{0}_p(M \subset N,\nabla^{\bot})$ is trivial and the claim follows.\\

Then the theorem follows from Lemma \ref{Helgason} taking into account that $\Hol^{0}_p(M \subset N,\nabla^{\bot})$ is a compact Lie group due to Olmos' holonomy theorem \cite{Ol90}.
\end{proof}

As an immediate consequence we have the following corollary.

\begin{cor} The normal holonomy group $\Hol^{0}_p(M,\nabla^{\bot})$ is a product $K \times SO(W)$ where $K \lhd \Hol^{0}_p(M \subset N,\nabla^{\bot}) $ is a compact normal subgroup containing the center of $\Hol^{0}_p(M \subset N,\nabla^{\bot})$. More precisely,  the normal subbundle $\nu_N M$ splits as \[ \nu_N M = \nu_1 \oplus \nu_2 \, \, , \] and the normal holonomy group $\Hol^{0}_p(M \subset N,\nabla^{\bot}) $ is a product \[ \Hol^{0}_p(M \subset N,\nabla^{\bot}) = K \times \rho(SO(W)) \, \, ,\] where $K$ acts on $\nu_1$ and $\rho(SO(W))$ acts on $\nu_2$.
The normal bundle $\nu_p M$ splits as \[\nu_p(M) = \nu_1 \oplus \nu_2 \oplus W \oplus W^{\bot}\]
and $\Hol^{0}_p(M,\nabla^{\bot}) = K \times SO(W) $ acts as $K \times \rho(SO(W)) \times SO(W) \times {\mathbf{1}}$.
\end{cor}

\begin{cor}
Let $M$ be a totally real submanifold of a complex space form $\mathbb{S}_c^{m}$, contained in a totally real totally geodesic submanifold $N$ of $\mathbb{S}_c^{m}$. The normal holonomy group $\Hol_p(M,\nabla^{\bot})$ acts as the holonomy representation of a Riemannian symmetric space if and only if the representation $\rho: SO(W) \to SO(\nu_2)$ is trivial. Moreover, if the codimension of $M$ in $N$ is smaller than $\mathrm{dim}(M)$, then $\Hol_p(M,\nabla^{\bot})$ acts as the holonomy representation of a symmetric space i.e. the representation $\rho$ is trivial.
\end{cor}

\begin{proof}
The first part is a direct consequence of the previous corollary.\\
For the second part we only need to prove that under these hypothesis $\rho(SO(W))$ is trivial.
Observe that $\mathrm{dim}(M) > 1$ and if $\mathrm{dim}(M) = 2$ then $M$ is a surface in the three dimensional real space form $N$. So $\nu_{N}M$ is flat hence $\rho$ is trivial.\\

 Assume $\dim (M)\geq 3$. Since the codimension of $M$ is smaller than its dimension, we should have a representation $\rho:SO(W)\to SO(\nu_2)$ with $\dim W\geq \dim J(TM) > \dim (\nu_2)$. This shows that $\rho$ is trivial in case $\dim W=3$ or $\dim W\geq 5$, since $SO(W)$ is a simple Lie group and $\dim SO(W) > \dim SO(\nu_2)$.

 If $\dim W=4$, we must have $3\leq \dim(M)\leq 4$. If $\dim(M)=4$, $J(TM)=W$ and hence $M$ is totally geodesic. So $\nu_{N}M$ is flat and $\rho(SO(W))$ is trivial. If $\dim(M)=3$, then it is a hypersurface in $N$ or it has codimension equal to $2$. In the last case, we should have a representation $\rho:SO(4)\to SO(2)$ which must be again trivial.
\end{proof}

\section{Complex submanifolds: Proof of Theorem \ref{teo:normalComplejo}}

In \cite{CDO} it was proved that the normal holonomy group of a full complete complex submanifold of the projective space is either the full group $SO(\nu(M))$  or the submanifold has parallel second fundamental form.

In this section we prove Theorem \ref{teo:normalComplejo} which improves the results of \cite{AD} and complete the local classification of normal holonomies of complex submanifolds of complex space forms. \\

For the flat complex space form $\mathbb{C}^n$ the result is in \cite[Remark 2.2., page 253]{Di00}.
Namely, the existence of a flat factor for the normal holonomy action implies a reduction of codimension which is not possible since the submanifold is assumed to be full.

For the complex projective space or its non compact dual the result was proved in \cite{AD} under the stronger hypothesis that either the action of $\Hol^0_p(M,\nabla^{\perp})$ is irreducible or the second fundamental form $\alpha$ of $M$ has no nullity i.e. the index of relative nullity $\mu(p)=\dim(\mathcal{N}_p)$ of $M$ is zero, where $$\mathcal{N}_p=\cap_{\xi\in \nu_pM}\ker(A_\xi).$$
So it is enough to show that if the index of relative nullity of $M$ is non zero then Theorem \ref{teo:normalComplejo}  holds.
In this case, there exists a unitary vector $X\in \mathcal{N}_p$ and then by the Ricci equation (\ref{ricci}), one has
\[ R^{\perp}(X,JX) \xi = -\frac{c}{2} J \xi \, \, ,\]
for any $\xi\in \nu_p M$, where $c$ is the constant holomorphic sectional curvature of the non flat complex space form.
This shows that the complex structure $J$ belongs to the Lie algebra of the normal holonomy group at the point $p \in M$. Then Theorem \ref{teo:normalComplejo} follows from \cite[Theorem 24
and Proposition 9]{AD}. $\hfill\Box$

\begin{rem}\label{rem:nonPolar}Without the hypothesis of the submanifold being full Theorem \ref{teo:normalComplejo}  is not true. Indeed, the normal holonomy group of a codimension 2 totally geodesic $\mathbb{C}P^n \subset \mathbb{C}P^{n+2}$ is the diagonal action of $U(1)$ on $\mathbb{C}^2$. Such action is not even polar \cite[pag. 92, exercise 3.10.6]{BCO} hence cannot be an s-representation.
\end{rem}

\begin{rem}\label{rem:nonIrreducible} The normal holonomy action of a complex submanifold is not necessarily irreducible. Here is an example:
Let $M \subset \mathbb{C}^3$ be the cone given by the equation \[ x^2 + y^2 + z^2 = 0 \, \, , \]
Then the projectivization $Z$ of the product of cones $M \times M \subset \mathbb{C}^6$ gives a 3-dimensional algebraic variety of $\mathbb{C}P^5$ . The normal holonomy group of the smooth open subset
 $ Z_{smooth} \subset Z$ does not act irreducibly on the normal space at any point $p \in Z_{smooth}$. This is so since the normal holonomy group of $Z_{smooth}$ is the same as the normal holonomy at a smooth point of the product $M \times M \subset \mathbb{C}^6$ see \cite[Remark 5, page 211]{CDO}.
 We note that the induced Riemannian metric (from the Fubini-Study metric on $\mathbb{C}P^5$) on $Z_{smooth}$ is locally irreducible.
\end{rem}


%
%
\vspace{.1cm}

\medskip

 A.J. Di Scala is member of PRIN 2010-2011 ``Varieta' reali e complesse: geometria,
topologia e analisi armonic'' and member of GNSAGA of INdAM.

\medskip

F. Vittone was partially supported by ERASMUS MUNDUS ACTION 2 programme, through the EUROTANGO II Research Fellowship, PICT 2010-1716 Foncyt and CONICET.

The second author would like to thanks Politecnico di Torino for the hospitality during his research stay.

\medskip

\vspace{1cm}

\begin{center}
\begin{tabular}{lcl}
 Antonio. J. Di Scala,&$\qquad$ & Francisco Vittone,\\
 \footnotesize Dipartimento di Scienze Matematiche&$\qquad$ &\footnotesize Depto. de Matem\'atica, ECEN, FCEIA,\\
 \footnotesize Politecnico di Torino,&$\qquad$ &\footnotesize Universidad Nac. de Rosario - CONICET\\
\footnotesize Corso Duca degli Abruzzi, 24&$\qquad$ &\footnotesize Av. Pellegrini 250\\
\footnotesize 10129 Torino, Italy &$\qquad$ &\footnotesize 2000, Rosario, Argentina \\
\footnotesize\href{mailto:antonio.discala@polito.it}{antonio.discala@polito.it}&$\qquad$ & \footnotesize\href{mailto:vittone@fceia.unr.edu.ar}{vittone@fceia.unr.edu.ar}\\
\footnotesize\url{http://calvino.polito.it/~adiscala/} &$\qquad$ &\footnotesize\url{www.fceia.unr.edu.ar/~vittone}
\end{tabular}
\end{center}

\end{document}